\documentclass[11pt]{amsart} \usepackage{latexsym, amssymb, stmaryrd}
\usepackage[T1]{fontenc}
\usepackage[multiple]{footmisc}

\newtheorem{thm}{Theorem}
\newtheorem{lemma}[thm]{Lemma}
\newtheorem{prop}[thm]{Proposition}
\newtheorem{cor}[thm]{Corollary}
\newtheorem{question}[thm]{Question}
\newtheorem{fact}[thm]{Fact}
\newtheorem{conj}[thm]{Conjecture}

\theoremstyle{definition}
\newtheorem{df}[thm]{Definition}

\newtheorem{rmk}[thm]{Remark}
\newtheorem{rmks}[thm]{Remarks}

\theoremstyle{remark}

\makeatletter

\def\dotminussym#1#2{%
  \setbox0=\hbox{$\m@th#1-$}%
  \kern.5\wd0%
  \hbox to 0pt{\hss\hbox{$\m@th#1-$}\hss}%
  \raise.6\ht0\hbox to 0pt{\hss$\m@th#1.$\hss}%
  \kern.5\wd0}
\newcommand{\dotminus}{\mathbin{\mathpalette\dotminussym{}}}
\renewcommand{\r}{\mathbb{R}}
\newcommand{\Z}{\mathbb{Z}}

\newcommand{\curly}[1]{\mathcal{#1}}

\newcommand{\B}{\curly{B}}

\renewcommand{\to}{\rightarrow}

\def \<{\langle}
\def \>{\rangle}

\def \*Z {{{^*}\Z}}
\def \((  {(\!(}
\def \)) {)\!)}

\def \ns{\operatorname{ns}}

\def \O{\operatorname{O}}

\numberwithin{equation}{section}

\def \R{\mathcal R}
\def \u{\mathcal U}

\def \O{\mathcal O}

\def \id{\operatorname{id}}
\def\cb{\operatorname{cb}}
\def\om{\omega}

\def \k{\mathcal{SF}}
\def \on{\operatorname{n}}
\def \os{\operatorname{s}}
\def \ons{\operatorname{ns}}

\newcommand{\cstar}{$\mathrm{C}^*$}

\allowdisplaybreaks[2]

\DeclareMathOperator{\ex}{ex}

\def \Q{\mathcal{Q}}

\title{Robinson Forcing and the Quasidiagonality Problem}
\author{Isaac Goldbring and Thomas Sinclair}
\thanks{I. Goldbring was partially supported by NSF CAREER grant DMS-1349399.}
\thanks{T. Sinclair was partially supported by NSF grant DMS-1600857.}
\thanks{The authors thank Bradd Hart, Mikael Rordam, Aaron Tikuisis, and Wilhelm Winter for useful conversations in connection to this work.}

\address{Department of Mathematics, Statistics, and Computer Science\\
University of Illinois at Chicago, Science and Engineering Offices (M/C 249) \\
851 S. Morgan St., Chicago, IL 60607-7045 and Department of Mathematics\\University of California, Irvine, 340 Rowland Hall (Bldg.\# 400),
Irvine, CA 92697-3875}
\email{isaac@math.uci.edu}

\address{Mathematics Department, Purdue University, 150 N. University Street, West Lafayette, IN 47907-2067}
\email{tsincla@purdue.edu}
\urladdr{http://www.math.purdue.edu/~tsincla/}

\begin{document}
\begin{abstract}
We introduce weakenings of two of the more prominent open problems in the classification of \cstar-algebras, namely the quasidiagonality problem and the UCT problem.  We show that the a positive solution of the conjunction of the two weaker problems implies a positive solution of the original quasidiagonality problem as well as allows us to give a local, finitary criteria for the MF problem, which asks whether every stably finite \cstar-algebra is MF.
\end{abstract}
\maketitle

%\begin{thm}
%Suppose that every stably finite nuclear C* algebra embeds in a nuclear quasidiagonal C* algebra.  Then the QDC holds.
%\end{thm}
%
%\begin{proof}
%If $A$ is e.c. for s.f. nuclear algebras and $B$ is a nuclear quasidiagonal algebra containing $B$, then since $A$ is e.c. in $B$ and being quasidiagonal is omitting types, $A$ is also quasidiagonal.
%\end{proof}
%
%\begin{cor}
%Suppose that quasidiagonality is preserved by cross product by $\z$.  Then the QDC holds.
%\end{cor}
%
%\begin{proof}
%Let $A$ be a stably finite nuclear algebra.  Let $B$ be the unitization of $C_0(\r,A)$.  Then $A\subseteq A\otimes \mathcal{K}\cong C_0(\r,A)\rtimes \z\subseteq B\rtimes \z$.  Now $B$ is quasidiagonal by a result of Voiculescu, whence so is $B\rtimes \z$ by assumption.  Also, $B$ is nuclear, whence so is $B\otimes \z$.  
%\end{proof}

\section{Introduction}

A prominent open problem in the theory of \cstar-algebras, raised in the work of Blackadar and Kirchberg \cite[section 7]{blackkir}, is whether or not every stably finite nuclear \cstar-algebra is quasidiagonal, what we will refer to as the \emph{quasidiagonality problem}.
%; the conjecture that the answer to this question is affirmative is often referred to as the \emph{Quasidiagonality conjecture} or \emph{Blackadar-Kirchberg conjecture}.  
Recent spectacular progress on this problem was made in \cite{TWW}, where it was shown that every separable, simple, unital, stably finite, nuclear \cstar-algebra in the so-called UCT class is quasidiagonal.  (UCT stands for ``Universal Coefficient Theorem'' and one can consult \cite{black} for more information on this class.)  Thus, a positive solution to the \emph{UCT problem}, which asks whether every separable nuclear \cstar-algebra is in the UCT class, yields a positive solution to the quasidiagonality problem for simple algebras.  In Section \ref{qdc}, we introduce weaker versions of the UCT problem and the quasidiagonality problem, and we prove that a positive solution to the conjunction of the two weaker problems implies a positive solution to the quasidiagonality problem in full.  The key model-theoretic notions involved in this proof are \emph{existentially closed algebras}, \emph{locally universal algebras}, and \emph{Robinson forcing}.  A brief review of these notions is the content of Section \ref{ecreview}.

Related to the quasidiagonality problem is the \emph{MF problem}, which asks whether every stably finite \cstar-algebra is MF, that is, embeds into an ultrapower of the universal, separable UHF algebra $\Q$.  By the Choi-Effros lifting theorem, a nuclear \cstar-algebra is quasidiagonal if and only if it is MF, whence it is readily seen that the MF problem is in fact a generalization of the quasidiagonality problem.  

In \cite{kep}, the current set of authors considered a problem of a similar nature to the MF problem, namely the \emph{Kirchberg embedding problem}, which asks whether every \cstar-algebra embeds into an ultrapower of the Cuntz algebra $\O_2$.  Using Robinson forcing, the authors were able to give a local, finitary reformulation of the Kirchberg embedding problem, namely that every \cstar-algebra has \emph{good nuclear witnesses} (see Definition \ref{gnw} below or \cite[Definition 3.6]{kep}).  It is reasonable to wonder whether or not a stably finite version of the above local criteria would similarly provide a reformulation of the MF problem.  In Section \ref{qdc}, we prove, once again assuming that both weak problems have positive solutions, that this is indeed the case, namely that the MF problem is equivalent to the problem of whether every stably finite \cstar-algebra has good stably finite nuclear witnesses.  

In \cite{kep}, the authors proved that the only possible nuclear \cstar-algebra that is existentially closed amongst all \cstar-algebras is $\O_2$ (and this is the case if and only if the Kirchberg embedding problem has a positive answer).  It is natural to wonder whether or not the corresponding statement is true in the stably finite situation, namely whether or not $\Q$ is the only possible nuclear stably finite \cstar-algebra that is existentially closed amongst all stably finite \cstar-algebras.  In Section \ref{sqdc}, we make some progress on this question by proving that a simple, nuclear stably finite \cstar-algebra in the UCT class that is existentially closed amongst all stably finite \cstar-algebras is AF.  Central to this result is the study of existentially closed subalgebras of II$_1$ factors, which is an interesting topic in its own right.  We also point out how these techniques also give some insight into the \emph{AF-embedding problem}, which asks whether every stably finite nuclear \cstar-algebra embeds into an AF algebra \cite[Question 7.3.3]{blackkir}.

The \emph{trace problem} asks whether every stably finite \cstar-algebra has a trace.  (We were unable to find a precise name for this problem in the literature so we have chosen to give it this ad hoc name here.)  Since MF algebras have a trace, a positive solution to the MF problem implies a positive solution to the trace problem.  Thus, assuming the aforementioned weak problems have a positive solution, the existence of good stably finite nuclear witnesses implies that every stably finite \cstar-algebra has a trace.  In Section \ref{trace}, by adapting work of Haagerup from \cite{traces}, we show directly, without any further assumptions, that the existence of good stably finite nuclear witnesses implies that every stably finite \cstar-algebra has a trace.  We also show that a quasitracial version of the notion of good stably finite nuclear witnesses yields a sufficient local condition for a positive solution of a question of Kaplansky, namely whether or not every quasitrace on a \cstar-algebra is actually a trace.

In order to keep this note fairly short, we make no attempt to keep the note self-contained.  A nice recent survey giving details on the context of this note is \cite{winter}.  We also assume that the reader has seen model-theoretic methods as applied to \cstar-algebras.  The monograph \cite{munster} has become the canonical reference; one can also consult our earlier paper \cite{kep}.

Let us conclude this section by introducting some notation.  
%First, \emph{all \cstar-algebras considered in this note are unital.}  
We let $\k$ denote the class of separable, stably finite, unital \cstar-algebras.  We let $\k_{\on}$ (resp.\ $\k_{\ons}$) denote the subclass of $\k$ consisting of the nuclear (resp.\ simple, nuclear) algebras.  To apply model-theoretic methods, it is important to note that $\k$ consists of the separable models of a universal theory of \cstar-algebras, whilst the classes $\k_{\on}$ and $\k_{\ons}$ are classes definable by uniform families of existential formulae (in the lingo from \cite{munster}).  

We use $\omega$ to denote an arbitrary nonprincipal ultrafilter on $\mathbb{N}$; given a family $(A_n)$ of \cstar-algebras indexed by $\mathbb{N}$, we write $\prod_\om A_n$ for the corresponding ultraproduct and $A^\omega$ in case $A_n=A$ for each $n$. 

Finally, as usual $\otimes$ with no further decoration denotes the minimal tensor product in the category of C$^*$-algebras, operator systems, or operator spaces.

\section{Existentially closed and locally universal algebras}\label{ecreview}

Recall that if $\theta:A\to B$ is a (unital) embedding between (unital) \cstar-algebras, then $\theta$ is said to be \emph{existential} if, for any quantifier-free formula $\phi(x,y)$ in the language of (unital) \cstar-algebras and any tuple $a$ from $A$, we have
$$\inf_{b\in A_1}\phi(a,b)=\inf_{b\in B_1}\phi(\theta(a),b).$$  If $A$ is a subalgebra of $B$ and the inclusion map is existential, we say that $A$ is \emph{existentially closed in} $B$.  If $A$ is existentially closed in all extensions belonging to a particular class of \cstar-algebras (e.g., the class of all stably finite \cstar-algebras), then we say that $A$ is existentially closed for that class.  

Before continuing further, we offer a perspective on existential embeddings which may seem more familiar to readers with a background in classification theory. Let $\theta: A\to B$ be an embedding of C$^*$-algebras. We say that $\theta$ is \emph{approximately split injective} if there exists a directed set $I$, an ultrafilter $\u$ on $I$, and an embedding $\sigma: B\to A^\u$ such that $\sigma\circ\theta: A\to A^\u$ equals the diagonal embedding of $A$ into $A^\u$. Note that if $B$ is separable, then one may take $\u$ to be a nonprincipal ultrafilter on the natural numbers. This notion was isolated and systematically studied in the context of C$^*$-algebras by Barlak and Szabo \cite{szabo} (not quite in the formulation given and with a different name), though it implicitly appears in the work of Gardella \cite{gardella} and the authors \cite{kep}. Similar notions appear even earlier, e.g., \cite[Definition 7.1.5]{blackkir}. It turns out that a $\ast$-embedding $\theta: A\to B$ is approximately split injective \emph{if and only if} $\theta$ is positively existential \cite[p.\ 170]{kep}. (See \cite[Theorem 4.19]{szabo} for one direction: the full equivalence in the category of operator systems is implicit in \cite[section 2.4]{kep} and \cite[section 5]{omit}.) Being positively existential is strictly weaker than being existential. However, one can easily check that all formulae referred to in this note are positive, whence the reader is free to replace ``existential'' with ``approximately split injective'' with little loss of essential meaning.  Note also that when we say that a map between \cstar-algebras is positively existential, we always mean this in the model-theoretic sense (as opposed to the map being existential and also a positive map between \cstar-algebras). 
%One significant caveat is that there seems to be no version of model-theoretic forcing that inputs merely positive existential formulae and produces positively existentially closed structures.

The key tool used throughout this paper is the construction of algebras using \emph{Robinson forcing}.  The details behind Robinson forcing (especially in the continuous case) can be quite cumbersome and we refer the interested reader to \cite{forcing}, \cite{farahmagidor}, and \cite[Appendix A]{kep} for complete treatments.  The main idea behind Robinson forcing is that one constructs separable algebras by specifying the operator norms of *-polynomial combinations of a distinguished countable dense set in a Baire category style fashion that allows one to construct these algebras ``slowly.''  For particularly nice classes, we can ensure that the ``generic'' algebra thus constructed has various desirable properties (e.g. nuclearity, simplicity, etc...) and is also existentially closed.

Suppose that $\mathcal{K}$ is a class of separable, unital \cstar-algebras.  We say that $A\in \mathcal{K}$ is \emph{locally universal for $\mathcal{K}$} if every element of $\mathcal{K}$ unitally embeds into an ultrapower of $A$.  Thus, the QD problem (resp. MF problem) asks whether or not $\Q$ is locally universal for $\k_{\on}$ (resp. $\k$).  

The following lemma is well-known and straightforward.  It is also implicitly used in \cite[Section 6.5]{munster}.

\begin{lemma}\label{jep}
Suppose that $A\in \mathcal{K}$ is existentially closed for $\mathcal{K}$.  Then the following are equivalent;
\begin{enumerate}
\item $A$ is locally universal for $\mathcal{K}$;
\item $\mathcal{K}$ has the \emph{joint embedding property}:  for every $B,C\in \mathcal{K}$, there is $D\in \mathcal{K}$ and unital embeddings of $B$ and $C$ into $D$.
\item for every $B\in \mathcal{K}$, there is $D\in \mathcal{K}$ and unital embeddings of $A$ and $B$ into $D$.
\end{enumerate}
\end{lemma} 

One can apply the previous lemma to the case $\mathcal{K}=\k_{\ons}$.  Indeed, $\k_{\ons}$ has an existentially closed element (see \cite{munster}).  Moreover, the tensor product of two elements of $\k_{\ons}$ belongs to $\k_{\ons}$ again by a result of Haagerup \cite{traces}.  Consequently, $\k_{\ons}$ has a locally universal element.

However, the situation with the class $\k$ (or $\k_{\on}$) seems to be somewhat more nebulous.  Indeed, whilst $\k$ also admits an existentially closed element, the following question seems to be open:

\begin{question}\label{jepsf}
Does the class $\k$ (or $\k_n$) have the joint embedding property?
\end{question}

Thus, at the moment, it is unclear whether or not there is any locally universal stably finite (nuclear) \cstar-algebra!

One approach to Question \ref{jepsf} might be to show that the minimal tensor product of two elements of $\k$ is once again an element of $\k$.  However, in the case both are simple this statement is equivalent to a positive solution to the trace problem again by a result of Haagerup (Fact \ref{hasatrace} below).  In fact, the following weaker question still seems to be open:

\begin{question}\label{tensornuclear}
Suppose that $A\in \k_{\on}$ and $B\in \k$.  Is $A\otimes B\in \k$?  
\end{question}  

For our purposes below, the following related lemma will prove useful. 

% (We thank Mikael Rordam for outlining this proof.)

\begin{lemma}\label{tensorMF}
Suppose that $A,B\in \k$ and $A$ is QD.  Then $A\otimes B$ is stably finite.  
\end{lemma}

\begin{proof}
Suppose $A\otimes B$ is not stably finite. It follows that there exists $v\in M_k(A\otimes B)$ such that $v^*v =1$ and $\|vv^* - 1\| = 1$. There is thus a u.c.p.\ map $\phi: A\to M_n$ such that for $\theta := \phi_k\otimes \id_B$ we have that $\theta(v)\in M_{kn}(B)$ satisfies $\|\theta(v)^*\theta(v) - 1\|< 1/4$ and $\|\theta(v)\theta(v)^* - 1\|> 3/4$. By functional calculus, there is $w\in M_{kn}(B)$ such that $w^*w = 1$ and $ww^*\not= 1$, contradicting that $B$ is stably finite.
\end{proof}

%\begin{lemma}\label{tensorMF}
%Suppose that $A,B\in \k$ and $B$ is MF.  Then $A\otimes B$ is stably finite.  (As usual, $\otimes$ with no further decoration denotes the %minimal tensor product.)
%\end{lemma}

%\begin{proof}
%First note the following chain of inclusions:
%$$A\otimes B\subseteq A\otimes \Q^\omega\subseteq (A\otimes \Q)^\omega.$$
%As $\Q$ is AF, it is straightforward to see that $A\otimes\Q$ is again stably finite. Therefore $(A\otimes\Q)^\omega$ is again stably finite, %whence so is $A\otimes B$.
%It is relatively straightforward to see that the lemma is true if one assumes that $B$ is AF, whence $A\otimes \Q$, and hence $(A\otimes \Q)^\omega$, are stably finite.
%\end{proof}

Note also that, by Lemma \ref{tensorMF}, we see that if the quasidiagonality problem has a positive solution, then Question \ref{tensornuclear} has a positive answer.

\section{The Quasidiagonality Problem}\label{qdc}
%Let $\k$ denote the class of unital stably finite nuclear \cstar-algebras.  Recall that the \emph{Quasidiagonality Conjecture} (QDC) states that every element of $\k$ is quasidiagonal, or, equivalently, that every element of $\k$ embeds into an ultrapower of $\Q$, where $Q$ is the universal UHF algebra.
%
%The following result (\cite[Corollary B]{TWW}) represents the greatest progress on the QDC thus far:
%
%\begin{fact}\label{UCTQDC}
%If $A\in \k$ is separable, simple and in the UCT class, then $A$ is quasidiagonal.
%\end{fact}

%The \emph{UCT conjecture} (UCTC) states that every separable nuclear C$^*$ algebra belongs to the UCT class.  In particular, the QDC follows from the UCTC.  The point of this note is to show that a weakened version of the QDC together with a weakened version of the UCTC imply the QDC.

We start this section by defining the aforementioned weakening of the UCT problem.

\begin{df}
The \emph{weak UCT problem} is the statement that every element of $\k_{\on}$ is contained in an element of $\k_{\on}$ that belongs to the UCT class.  
%Likewise, one can define the weak UCT problems for the classes $\k_{\on}$, $\k_{\os}$, and $\k_{\ons}$.
\end{df}

%Wilhelm suggested to me that this remark should be removed.
%\begin{rmk}
%(Due to Aaron Tikuisis) Minor evidence for the weak UCT conjecture for $\k_n$ is the fact that every element of $\k_n$ is KK-equivalent to an algebra that satisfies the weak UCT conjecture for %$\k_n$.  Indeed, if $A$ belongs to $\k_n$, then $A$ is KK-equivalent to its suspension $SA$, which belongs to $\k_n$ and is AF-embeddable by a result of Ozawa \cite{ozawa}.
%\end{rmk}

\begin{df}
The \emph{weak quasidiagonality problem} is the statement that there is a simple locally universal element of $\k_{\on}$.
% given any $A\in \k$, there is a \emph{simple} $B\in \k$ such that $A$ embeds into an ultrapower of $B$.  
 %Similarly, one can define the weak quasidiagonality problem for $\k_{\on}$.
\end{df}

%Note that the MF problem (resp.\ quasidiagonality problem), if resolved in the positive, predicts that the simple algebra $B$ in the statement of the weak quasidiagonality problem for $\k$ (resp.\ $\k_{\on}$) is always $\Q$.  
Note that in the statement of the weak quasidiagonality problem, we really do need to use an ultrapower of $B$ as there are stably finite, (nuclear) unital \cstar-algebras that are not contained in a simple, stably finite, (nuclear) unital \cstar-algebra.
%\begin{rmks}
%
%\
%
%\begin{enumerate}
%\item The weak QD conjecture is indeed implied by the QD conjecture as $Q$ is simple.
%\item The version of the weak QD conjecture where one asks for $A$ to embed into $B$ (as opposed to an ultrapower of $B$) is simply false.
%\item Minor evidence towards the truth of the weak UCT conjecture is that every element of $\k$ is KK-equivalent to an element of $\k$ that satisfies the weak UCT conjecture.
%\end{enumerate}
%\end{rmks}

Here is the main result of this section.

\begin{thm}\label{weakconjunction}
Suppose that both the weak UCT problem and weak quasidiagonality problem have positive solutions.  Then the quasidiagonality problem has a positive solution.
\end{thm}

\begin{proof}
%By Lemma \ref{jep}, it is enough to prove that there is $A\in \k_n$ that is existentially closed for $\k_{\on}$ and is quasidiagonal.  
By Lemmas \ref{jep} and \ref{tensorMF}, it is enough to prove that there exists $A\in \k_{\on}$ that is existentially closed for $\k_{\on}$ and quasidiagonal.  The version of the Omitting Types theorem via Robinson forcing (see \cite[Corollary 4.7]{forcing}), together with the assumption that the weak quasidiagonality problem has a positive answer and the fact that simplicity is definable by a uniform family of existential formulae (see \cite[Theorem 5.7.3(6)]{munster}), allows us to construct a separable algebra $A$ that is simple and ``finitely generic'' for $\k_{\on}$.  (Finitely generic simply means that the structure is obtained from finite forcing as opposed to infinite forcing.)  The discussion in \cite[Appendix A]{kep} shows that if $\k_{\on}$ were the class of models of an $\forall\exists$-theory, then structures that are finitely generic for $\k_{\on}$ belong to $\k_{\on}$ and are existentially closed for $\k_{\on}$.  In our case, $\k_{\on}$ is not the class of models of an $\forall\exists$-theory, but rather is a class of structures definable by a uniform family of existential formulae.  It is readily verified that the proof given in \cite[Appendix A]{kep} goes through in this setting as well (see also \cite[Corollary 4.6]{farahmagidor}), whence the separable, simple \cstar-algebra $A$ constructed belongs to $\k_{\on}$ and is existentially closed for $\k_{\on}$.  Since we are assuming that the the weak UCT problem has a positive solution, $A$ is a subalgebra of a separable $B\in \k_{\on}$ that belongs to the UCT class.  Since $A$ is existentially closed for $\k_{\on}$, by \cite[Theorem 2.10]{szabo} (see also \cite[Theorem 3.13]{gardella}), $A$ also belongs to the UCT class.  By the main result of \cite{TWW} referred to in the introduction, $A$ is quasidiagonal.
%By the positive solution to the weak quasidiagonality problem, we may fix $B\in \k_{\ons}$ that is locally universal for $\k_{\on}$.  Then $A\otimes B$ is stably finite and nuclear (as both $A$ and $B$ are simple), whence $A$ is existentially closed in $A\otimes B$.  It follows that $A\otimes B$ (unitally) embeds into $A^\omega$, whence so does $B$.  In particular, $A$ is locally universal for $\k_{\on}$.  Since $A$ is quasidiagonal, this completes the proof.
\end{proof}

\begin{rmk}
To run the preceding proof, one does not really need the full strength of the weak quasidiagonality problem for $\k_{\on}$.  Indeed, for each integer $n$, let $\varphi_n(a,b)$ denote the formula
$$\varphi_n(a,b):=\inf_{\{x\in A^n \ : \ \|\sum x_j^*x_j\|\leq 2\}}\|\sum x_j^*ax_j-b\|.$$  It is shown in \cite[Proposition 5.10.3]{munster} that $A$ is simple if and only if, for each $a,b\in A$, we have $\inf_n \varphi_n(a,b)=0$.  Consequently, to be able to apply the Omitting Types Theorem as in the proof of the previous theorem, one needs to be able to verify the following:  for each $A\in \k_{\on}$ and each $a_1,\ldots,a_m\in A$, one can find $B\in \k_{\on}$ satisfying the UCT and $b_1,\ldots,b_m\in B$ whose operator norm microstates are close to those of $a_1,\ldots,a_m$ and for which $\inf_n\varphi_n(b_i,b_j)$ is small for each $i,j$.  (Of course the $B$ and $b_1,\ldots,b_m$ depend on the level of precision of the microstates and the tolerance for $\inf_n \varphi_n$.)
\end{rmk}

\begin{rmk}
The situation is simpler if one works in $\k_{\ons}$.  Indeed, let $A$ be any element of $\k_{\ons}$ that is existentially closed for $\k_{\ons}$.  Then a positive solution to the weak UCT problem implies that $A$ is in the UCT class, whence quasidiagonal.  (Note here we only need the weak UCT problem to have a positive solution for simple algebras.)  It follows that all elements of $\k_{\ons}$ are quasidiagonal.  In other words, a positive solution to the weak UCT problem implies what Winter calls QDQ$_{\operatorname{simple},1}$ in \cite{winter}.  As Winter points out, it is plausible that QDQ$_{\operatorname{simple},1}$ could imply  a positive solution to the quasidiagonality problem.
\end{rmk}

%\begin{rmk}
%Winter asked whether or not QDC$_{\text{simple,1}}$ implies QDC.  Note that it does assuming the weak QD conjecture.  Indeed, assuming the weak QD conjecture, one can construct $A\in \k$ that is e.c. for $\k$ and simple.  By QDC$_{\text{simple,1}}$, $A$ is QD, whence the QD conjecture holds.
%\end{rmk}
We now turn to the local, finitary equivalent of the MF problem  First, we recall some terminology from \cite{kep}.  Recall that for an operator system $E$, $\ex(E):=\inf_X d_{\cb}(E,X)$, where $X$ ranges over all matricial operator systems and $d_{\cb}$ is the completely bounded version of the Banach-Mazur distance.  (This is not literally the definition given in \cite{kep} but is readily seen to be equivalent.)
\begin{df}\label{gnw} We say that a unital \cstar-algebra $A$ has \emph{good nuclear witnesses} if, for every finite-dimensional operator system $F\subset A$, there is a $\ast$-embedding $\rho:=(\rho_n)^{\bullet}:A\to \B(H)^{\om}$ such that $\ex(\rho_n(F))\to 1$ as $n\to \om$.
 \end{df}

In the context of this paper, it is natural to extend the above definition to the stably finite context.  Indeed, if in the notation above each $F_n$ is isomorphic (as an operator system) to a subsystem of a \cstar-algebra $Q_n$ that belongs to $\k$ (respectively $\k_{\os}$), then we say that $A$ has \emph{good (resp.\ simple) stably finite nuclear witnesses}.  Note that if $A$ has good stably finite nuclear witnesses, then $A$ is itself stably finite.

\begin{cor}\label{MF}
Suppose that the weak UCT problem and weak quasidiagonality problem both have positive solutions.  Then the MF problem is equivalent to the statement that every element of $\k$ has good stably finite nuclear witnesses.
\end{cor}

\begin{proof}
As in \cite{kep}, the forward direction is immediate (and does not use either of the weak problems).  Now assume that every element of $\k$ has good stably finite nuclear witnesses.  Then by Robinson forcing, we can find an existentially closed element $A$ of $\k$ that is also nuclear.  (See \cite[Section 3.2]{kep} for more details.)  By Theorem \ref{weakconjunction}, we know that $A$ is quasidiagonal.  By Lemmas \ref{jep} and \ref{tensorMF}, $A$ is locally universal for $\k$, whence the MF problem has a positive solution.  
%Now suppose that $B$ is an arbitrary element of $\k$.  Then by Lemma \ref{tensorMF}, $A\otimes B$ is stably finite, whence $A$ is existentially closed in $A\otimes B$.  In particular, $A\otimes B$ is $A^\omega$-embeddable, thus MF.  It follows that $B$ is MF, as desired.  
\end{proof}
   
\section{Nuclear existentially closed elements of $\k$}\label{sqdc}

In \cite{kep}, it is shown that the only possible \cstar-algebra that is nuclear and existentially closed amongst all \cstar-algebras is $\O_2$.  It is natural to wonder whether or not there is a stably finite version of this result.  Here is the natural guess:

\begin{conj}
If $A\in \k$ is existentially closed for $\k$ and nuclear, then $A\cong \Q$.
\end{conj}

In this section, we make some progress towards settling the previous conjecture.  We begin by enumerating a list of general properties that hold of a unital \cstar-algebra that admits an existential (unital) embedding into a II$_1$ factor. Before doing so, we pause to recall a few standard definitions and results.

\begin{df}\label{defqt} A \emph{quasitrace} (really, a \emph{2-quasitrace}) on a \cstar-algebra $A$ is a function $\tau:A\to \mathbb{C}$ which satisfies:
\begin{enumerate}
\item $\tau(xx^*)=\tau(x^*x)\geq 0$ for all $x\in A$;
\item $\tau(a+ib)=\tau(a)+\tau(b)i$ for all self-adjoint $a,b\in A$;
\item $\tau(ab)=\tau(ba)$ for all self-adjoint and commuting $a,b\in A$;
\item there is a function $\tau_2:M_2(A)\to \mathbb{C}$ satisfying (1)-(3) above such that $\tau(x)=\tau_2(x\otimes e_{11})$ for all $x\in A$.
\end{enumerate}
\end{df}

\noindent By \cite[Proposition II.4.1]{blackhand} any quasitrace $\tau$ on $A$ extends to a quasitrace $\tau_n$ on $M_n(A)$ for all $n\in \mathbb N$. For a positive element $a\in M_k(A)$ and a quasitrace $\tau$ define the dimension function by $d_\tau(a) := \lim_{n\to\infty} \tau_k(a^{1/n})$. The C$^*$-algebra $A$ is then said to have \emph{strict comparison} if for all positive $a,b\in M_k(A)$ such that $d_\tau(a)< d_\tau(b)$ for all quasitraces $\tau$ on $A$ there is a sequence $r_n\in M_k(A)$ such that $\|r_n^*br_n - a\|\to 0$. The reader may consult \cite{kiro} for instance for the history and context of comparison properties in the classification of stably finite nuclear C$^*$-algebras. 

We are now ready to state and prove the main proposition of this section. As we will remark, some statements are known to hold by alternate means. However, we wish to illustrate how through the use of an existential embedding into a von Neumann algebra, one can simply and directly transfer von Neumann algebraic comparison theory to the embedded C$^*$-algebra, thereby avoiding many of the difficulties one confronts in working with comparison theory at the level of C$^*$-algebras. In the future, one could hope that model theoretic techniques could be used to ``regularize'' a stably finite C$^*$-algebra by embedding it into a suitably ``nice'' C$^*$-algebra with much better comparison properties.

\begin{prop}\label{ecinafactor}
Suppose that $A$ is a separable, unital \cstar-algebra that admits an (positive) existential unital embedding into a II$_1$ factor $M$.  Then:
%Suppose that $A$ is a separable \cstar-algebra that admits an existential embedding into the hyperfinite II$_1$ factor $\R$.  Then:
\begin{enumerate}
\item $A$ is simple, monotracial, and the unique trace is definable;
\item there is a simple, monotracial AF-algebra $B$ such that $K_*(A)\cong K_*(B)$.
\item for any two projections $p, q\in A$ either $p\preceq q$ or $q\preceq p$ (``$\preceq$'' denotes Murray-von Neumann subequivalence);
\item $A$ has unique quasitrace;
\item $A$ has strict comparison.
\end{enumerate}
\end{prop}

\begin{proof}
As in \cite[Proposition 2.5]{kep}, we have that $A$ has the uniform Dixmier property in the parlance of \cite[section 7.2]{munster}, whence it is simple and definably monotracial by \cite[Lemma 7.2.2]{munster}, proving (1).  In fact, the analysis given in \cite{munster} shows that the trace is existentially definable and the embedding of $A$ into $M$ is existential in the language of \emph{tracial} \cstar-algebras, that is, the language of \cstar-algebras enlarged with a predicate symbol naming the trace. Now since the embedding of $A$ in $M$ is existential, the induced map on K-theory is injective. (See, for instance, \cite[section 2.3]{kep} or \cite[Theorems 2.8 and 4.19]{szabo}.) In particular, $K_1(A)=0$.  Since the embedding of $A$ into $M$ is also existential in the language of tracial \cstar-algebras, the induced injection $K_0(A)\to K_0(M)=\r$ is an embedding of scaled ordered groups.  Moreover, once again using existentiality in the language of tracial \cstar-algebras, we see that the image of the embedding is dense in $\r$.   
%Further, it can be checked by existentiality that the induced embedding $K_0(A)\to \r (=K_0(\mathcal R))$ preserves the dimension group structure and has dense range.  
It follows from the Effros-Handelman-Shen Theorem \cite{ehs} (more precisely, the version stated in \cite[Theorem 7.4.3]{black}) that $K_0(A)\cong K_0(B)$ as scaled ordered groups for some simple, unital monotracial AF algebra $B$, proving (2). 

For (3), fix projections $p,q\in A$.  Without loss of generality, we may assume that $p\preceq q$ in $M$; we show that $p\preceq q$ in $A$.  Let $\phi(x,y)$ be the formula $\inf_z\max(\|x-z^*z\|,\|y-zz^*\|)$.  Since $p\preceq q$ in $M$ and $A$ is existentially closed in $M$, we have
$$(\inf_{p'}\max(\|p'q-p'\|,\phi(p',p)))^A=0.$$  In the above formula, the infimum ranges over projections.  Fix $\epsilon>0$ sufficiently small and let $\delta=\Delta_\phi(\epsilon),$ where $\Delta_\phi$ is the modulus of uniform continuity for $\phi$.  Let $\eta>0$ be sufficiently small and take a projection $p'\in A$ such that $\max(\|p'q-p'\|,\phi(p',p))<\eta$.  Then if $\eta$ is sufficiently small, there is a projection $p''\in A$ with $p''\leq q$ and $\|p''-p'\|<\delta$, whence $\phi(p'',p)<\epsilon$; if $\epsilon>0$ was chosen sufficiently small, then we have that $p''$ is Murray-von Neumann equivalent to $p$, witnessing that $p\preceq q$.
%consider the formulas $\phi(v, p, q) = \|p - v^*v\| + \|vpv^*q - vpv^*\|$ and $\psi(p,q) = \inf_v\min\{\phi(v,p,q),\phi(v,q,p)\}$. Clearly $\phi(v,p,q)=0$ if and only if the partial isometry $v$ witnesses $p\preceq q$. We need two standard facts:  first, if $p,q\in A$ are projections and $v$ is a partial isometry such that $\|p - v^*v\| + \|vpv^* - q\|< 1$, then there is $w\in A$ such that $p = w^*w$ and $q = ww^*$; and second, given $\epsilon>0$, there is $\delta>0$ such that if $p,q\in A$ are projections with $\|pq - p\|<\delta$ then there is $p'\leq q$ such that $\|p - p'\|<\epsilon$. The conclusion of (3) now follows as partial isometries are definable and $\psi(p,q)^A=0$ by existentiality.

By part (3) and the density of the trace values on projections in $A$, the standard argument for showing that the trace is the unique dimension function on the projection lattice of $M$ applies, and we know that $\phi(p) = \tau(p)$ for any projection $p\in A$ and any quasitrace $\phi$ on $A$. Since quasitraces are norm continuous and are almost linear on almost commuting elements of $A$ by \cite{blackhand} and any self-adjoint contraction in $M$ may be $1/n$-norm approximated by a linear combination of $2n$ orthogonal projections in $M$, the conclusion of (4) follows.

For (5), let $a,b\in M_n(A)_+$ be contractions; by part (4), we only need to show that if the support projection of $a$ is subequivalent to the support projection of $b$, then there exists a sequence $r_n\in A$ such that $\|r_n^* b r_n - a\|\leq 1/n$ for all $n$. We can easily find such a sequence in $M$ (even with $\|r_n\|\leq n$), so we are done by existentiality. 
\end{proof}

\begin{rmks} 

\

\begin{enumerate}
\item The proof of statement (3) boils down to the quite useful fact that Murray-von Neumann equivalence is existential and weakly stable in the language of C$^*$-algebras.
\item By the main result of \cite{traces}, statement (4) is immediately implied by statement (1) in the case that $A$ is exact. 
\item In \cite[Theorem 8.2.1]{munster}, it is shown that strict comparison is $\forall\exists$-axiomatizable, yielding a different proof of item (5) in the previous proposition.
\end{enumerate}
\end{rmks}
%
%The following is noted to be a corollary of the main result of \cite{TWW} even without the assumption of existential closure (see Corollary 6.5 therein); however, the derivation of this fact in the e.c.\ case is considerably simpler.

\begin{cor}\label{AF}
Suppose that $A\in \k_{\ns}$ is in the UCT class and existentially closed for $\k$.  Then $A$ is AF.
\end{cor}

%\begin{prop}
%Suppose that $A$ is separable, simple, unital, nuclear, in the UCT class and e.c. for the class of s.f. algebras.  Then $A$ is AF.
%\end{prop}

\begin{proof}
%Since $A$ is separable, simple, and nuclear, it admits a faithful trace by \cite{traces}, whence we have an embedding of $A$ into $\R$, the hyperfinite II$_1$ factor. Moreover, we may conclude, similarly to \cite[Proposition 2.5]{kep}, that $A$ has the uniform Dixmier property in the parlance of \cite[section 7.2]{munster}, whence is (definably) monotracial. Since $A$ is e.c., this induces an embedding of K-theory. (See, for instance, \cite[section 2.3]{kep} or \cite[Theorems 2.8 and 4.19]{szabo}.) In particular, $K_1(A)=0$.  Further, it can be checked by existentiality that the induced embedding $K_0(A)\to \r (=K_0(\mathcal R))$ preserves the dimension group structure and has dense range.  It follows from \cite{ehs} (see \cite[Theorem 7.4.3]{black}) $K_0(A)\cong K_0(B)$ as ordered groups for some simple, unital monotracial AF algebra $B$.  
We have that $A$ is separable, simple, and nuclear, whence $A$ admits a faithful trace by \cite{traces}. It follows that the von Neumann algebraic closure of $A$ with respect to the trace is semidiscrete, whence is hyperfinite by Connes' classification of injective factors \cite{connes}. Thus we have a unital embedding of $A$ into $\R$ (the hyperfinite II$_1$ factor), which is moreover existential since $A$ is existentially closed for $\k$.  Let $B$ be the simple, monotracial AF algebra given by the previous proposition.  By \cite[Corollary D]{TWW}, we can conclude that $A\cong B$ since we know that $A$ is also of finite nuclear dimension, which follows from item (5) in Proposition \ref{ecinafactor} and the resolution of the Toms-Winter conjecture in the monotracial case \cite{sww}.
%we just need $A$ to be $\mathcal{Z}$-stable, which it is as it is existentially closed for $\k$ (see \cite[Proposition 5.12]{kiro} or \cite[Section 4.3]{munster}).
\end{proof}

%In an appendix written by Dominic Enders, it is shown that, under the assumptions of the previous corollary, $A$ must in fact be isomorphic to $\mathcal{Q}$.

%In \cite[Proposition 2.18]{kep}, it is shown that if $A$ is an existentially closed \cstar-algebra that is also nuclear, then $A\cong \O_2$.  We can view Corollary \ref{AF} as mild evidence towards the stably finite version of this statement:
%
%\begin{question}
%If $A$ satisfies the hypotheses of Corollary \ref{AF}, must $A$ be isomorphic to $\Q$?
%\end{question}

\begin{rmk}
By \cite[Corollary 6.5]{TWW}, if $A\in \k_{\ons}$ is monotracial and UCT, then $A$ embeds into a simple, monotracial AF algebra.  Since being AF is definable by a uniform family of existential formulae \cite[Theorem 2]{omitting}, an existential substructure of an AF algebra is AF again, whence we obtain a different proof of Corollary \ref{AF}.  However, we prefer the above proof as it follows from the main result of \cite{TWW} in a more elementary fashion.
\end{rmk}

\begin{question}
Is there a proof of Corollary \ref{AF} that avoids the use of the ideas in \cite{TWW}?
\end{question}

\begin{rmk}
A positive solution to the previous question could lead to an alternate approach to (the nontrivial direction of) Corollary \ref{MF} that is more elementary in nature.  Indeed, from good stably finite nuclear witnesses, we know that there is $A\in \k$ that is existentially closed for $\k$ and nuclear.  If Question \ref{tensornuclear} has a positive answer, then by Lemma \ref{jep}, $A$ is locally universal for $\k$.  By a positive solution to the weak QD problem, it follows that there is $C\in \k_{\ons}$ such that is locally universal for $\k$.  One can now apply Robinson forcing again to conclude that there is $D\in \k$ that is existentially closed for $\k$ that is also \emph{simple} and nuclear.  By a positive solution to the weak UCT problem, $D$ is also UCT, whence $D$ is AF by Corollary \ref{AF}.  By Lemma \ref{jepsf}, $D$ is locally universal for $\k$, whence the MF problem has a positive solution.  
\end{rmk}

In \cite{kep}, it is shown that $\O_2$ is the only possible existentially closed \cstar-algebra that is a tensor square.  We can (assuming weak UCT) establish the stably finite version of this statement:

\begin{cor}
Assume that the weak UCT problem has a positive solution.  Suppose that $A$ is existentially closed for $\k$ and that there is a \cstar-algebra $B$ such that $A\cong B\otimes B$.  Then $A\cong \Q$ (whence the MF problem has a positive solution).  If, in addition, $B$ itself is also existentially closed for $\k$, then $B\cong \Q$.
\end{cor}

\begin{proof}
Let $\alpha:A\to A$ denote the flip automorphism, namely $\alpha(a\otimes b)=b\otimes a$ for all $a,b\in B$.  Note that the associated crossed product C$^*$-algebra, $A\rtimes_\alpha \mathbb{Z}_2$,  is again stably finite as it can be realized as a subalgebra of $M_2(A)$. Since $A$ is existentially closed for $\k$, it follows that $\alpha$ is approximately inner (see \cite[Section 3]{ecfactor}), whence $B$ has \emph{approximately inner flip}.  

At this point, there are a couple of (more or less equivalent) ways to conclude that $A\cong \Q$.  One way is to note that $A\cong B\otimes B$ also has approximately inner flip (see \cite{er}), so is simple and nuclear (again, see \cite{er}).  By the weak UCT, Corollary \ref{AF} implies that $A$ is AF.  In \cite{er}, it is shown that an AF algebra with approximately inner flip is UHF.  Since $A$ is existentially closed amongst the UHF algebras, we must have $A\cong \Q$.  

Alternatively, by \cite[Proposition 1.9(ii)]{ssa}, we have that $B^{\otimes \infty}$ (the tensor product of $B$ with itself countably many times) is \emph{strongly self-absorbing} (see \cite{ssa} for the definition).  However, in \cite[Theorem 2.9(6)]{szabo}, it is shown that if $D$ is strongly self-absorbing, then any algebra existentially closed in $D$ must be isomorphic to $D$.   Since $A\cong B\otimes B$ is existentially closed in $B^{\otimes \infty}$, we can conclude that $A\cong B^{\otimes \infty}$.  By the positive solution to the weak UCT problem, $A$ satisfies UCT, whence $A\subseteq \Q$ as \cite{TWW} implies that all strongly self-absoring \cstar-algebras satisfying the UCT embed in $\Q$.  Finally, since $A$ is existentially closed in $\Q$, we can conclude that $A\cong \Q$.

If $B$ is existentially closed for $\k$, then since $B\subseteq A\cong \Q$, we can conclude that $B\cong \Q$.
\end{proof}

\begin{rmk} The use of \cite[Theorem 2.9(6)]{szabo} in the proof above allows us the opportunity to present a model-theoretic proof of that result. Indeed, suppose that $D$ is strongly self-absorbing and $E$ is existentially closed in $D$.  Fix an embedding $D\to E^\om$ that restricts to the diagonal embedding on $E$.  By considering the chain
$$E\subseteq D\to E^\omega \subseteq D^\omega \to (E^{\omega})^{\omega}\cdots$$ and using the fact that the maps between the successive ultrapowers of $E$ are elementary (being ultrapowers of the diagonal map) and the maps between the successive ultrapowers of $D$ are elementary (as the initial map $D\to D^\om$ is elementary as $D$ is strongly self-absoring), we see that $D$ and $E$ are elementarily equivalent.  It follows that the map $D\to E^\om$ is elementary (again, since $D$ is strongly self-absorbing), whence $E$ is actually elementary in $D$. 
 Since $D$ is strongly self-absorbing, it is the prime model of its theory, whence so is $E$.  By uniquness of prime models, we see that $D\cong E$.
 
 We point out that this result gives a new proof of the aforementioned result of the authors that any separable, nuclear C$^*$-algebra, existentially closed amongst all \cstar-algebras, is $\ast$-isomorphic to $\O_2$.  This proof has the advantage of only relying on the Kirchberg-Phillips embedding theorem, which states separable nuclear \cstar-algebra embeds in $\O_2$, rather than the more difficult ``$A\otimes \O_2$'' theorem of Kirchberg (as well as the axiomatizability of simple, purely infinite algebras).
\end{rmk}

We conclude this section by making a remark on the \emph{AF-embeddability problem}, which asks whether every element of $\k_{\on}$ embeds into an AF algebra.  (See, for example, \cite[Section 8.5]{BO}.)  By Corollary \ref{AF}, we immediately have:

\begin{cor}
Suppose that $B$ is an element of $\k_{\on}$ that embeds into an element of $\k_{\ons}$ that is in the UCT class and that is existentially closed for $\k$.  Then $B$ is AF-embeddable.
\end{cor}

By relativizing the above concepts to a fixed element $B$ of $\k_{\on}$ and using the preceding corollary, we can derive a sufficient local, finitary criteria for the AF-embeddability problem (for simple algebras).  If $A$ is an element of $\k$ that contains $B$, we say that $A$ has \emph{good simple stably finite nuclear witnesses over $B$} if the algebras $Q_n$ are simple, contain $B$, and the embedding $\rho:B\to \prod_\om Q_n$ restricts to the diagonal embedding on $A$.  
%By a \emph{$B$-condition} we now mean a condition with parameters from $B$; a $B$-condition is satisfiable if it is satisfied by a tuple in a C$^*$ algebra extension of $B$.  We say that a $B$-condition $p(x)$ has \emph{good stably finite nuclear witnesses} if, for every $\epsilon>0$, there is a stably finite C$^*$ algebra $C$ extending $B$ and a tuple $c$ from $C$ satisfying $p(x)$ such that $\Delta_{\nuc}(c)<\epsilon$.  
%In addition, the \emph{weak quasidiagonality problem over $B$} asks whether, given any $C\in \k_{\on}$ extending $B$, there is $D$ in $\k_{\ons}$ extending $B$ and an embedding $\iota:C\to D^\om$ such that $\iota(b)=\Delta(b)$, where $\Delta:D\to D^\om$ is the diagonal embedding.

%\begin{cor}
%Suppose that $B$ belongs to $\k_{\ons}$.  Assume that the weak UCT problem for $\k$ has a positive solution, and suppose that the following hold:
%\begin{enumerate}
%\item Every element of $\k$ that contains $B$ has good stably finite nuclear witnesses over $B$.
%\item The weak quasidiagonality problem over $B$ has a positive solution.
%\end{enumerate}
%Then $B$ is AF-embeddable.
%\end{cor}

\begin{cor}
Assume that the weak UCT problem has a positive solution.  Suppose that $B\in \k$ is such that every element of $\k$ containing $B$ has good simple stably finite nuclear witnesses over $B$.  Then $B$ is AF-embeddable.
\end{cor}

\begin{proof}
The hypotheses allow us to run the model-theoretic forcing machinery in the language expanded by constants to name elements of $B$.  The result is a unital \cstar-algebra $C$ that is existentially closed for the elements of $\k$ that contain $B$ that is also simple and nuclear.  In particular, $C$ is existentially closed for $\k$.  By the positive solution of the weak UCT problem for $\k$, $C$ is also UCT.  It follows that $C$ is AF.
\end{proof}

\section{Traces on stably finite \cstar-algebras}\label{trace}

In this section, we show that a positive solution to the trace problem follows from every stably finite \cstar-algebra admitting good stably finite nuclear witnesses without needing any weak UCT or weak QD assumptions.  The following result is the key technical modification of Haagerup's work that we will need to accomplish this.

\begin{prop}\label{stablyfinitetensor} Suppose that $A$ and $B$ are separable, unital, simple \cstar-algebras, $A$ has good stably finite nuclear witnesses, $B$ is exact, and $\rho = (\rho_n)^\bullet: B\to \prod^\om M_n$ is a $\ast$-embedding of $B$ into a \emph{tracial} ultraproduct of matrix algebras. Then $A\otimes B$ is stably finite.
\end{prop}

\begin{proof} 
%Suppose that $A_n$ is an increasing family of C$^*$-subalgebras of $A$ such that $\bigcup_n A_n$ is dense. Then $A\otimes B$ is stably finite if and only if $A_n\otimes B$ is stably finite for each $n$. Consequently, we may thus assume for the sake of simplicity that $A$ is generated by a finite collection of unitaries. 
Fix an operator subsystem $F$ of $A$ spanned by finitely many unitaries. Let $F_k\subset Q_k$ be a sequence of pairs of operator systems and stably finite C$^*$-algebras as guaranteed by good stably finite nuclear witnesses for $F\subset A$ such that ${\rm ex}(F_k)\leq 1+1/k$.  For each $k$, we have a finite AW$^*$-algebra $N_k:=\prod^\om (Q_k\otimes M_n)$, where the ultraproduct is taken with respect to the quasitrace ideal; see \cite[Proposition 4.2 and Lemma 5.6]{traces}.  It follows from \cite{pis} that $\tilde\rho_k: F_k\otimes B\to N_k$ satisfies $\|\tilde\rho_k\|_{\cb}\leq 1+1/k$.
%By simplicity $Q_k$ has a faithful quasitrace. It follows from \cite{pis} that for $\tilde\rho_k: F_k\otimes B\to \prod(Q_k\otimes M_n)/(1\otimes J_\om) =: N_k$ we have that $\max\{\|\tilde\rho_k\|_{\cb}, \|\tilde\rho_k^{-1}\|_{\cb}\}\leq 1+1/k$. By \cite[Propostion 4.2 and Lemma 5.6]{traces}, $N_k$ is a finite AW$^*$-algebra with faithful quasitrace, whence is stably finite. 
By the exactness of $B$, $\theta: F\otimes B\to \prod_\om (F_k \otimes B)$ is a $\cb$-isometry \cite[Theorem 17.7]{pis}. We now easily see that $\rho := (\tilde\rho_k)^\bullet\circ \theta: F\otimes B\to \prod_\om N_k$ is a u.c.p.\ map which maps unitaries to unitaries.  By considering larger finite sets of unitaries there is a u.c.p.\ map $\rho': F'\otimes B\to \prod_{\om'} N_k$ which maps unitaries to unitaries where $F'\otimes B$ is an operator system which generates $A\otimes B$ as C$^*$-algebra. Pisier's linearization trick \cite[Proposition 13.6]{pis} states that we can extend any such u.c.p.\ map $\rho'$ to a $*$-homomorphism $A\otimes B\to \prod_{\om'} N_k$.  Since $A\otimes B$ is simple, the aforementioned $*$-homomorphism is injective, whence we have the desired conclusion since $\prod_{\om'} N_k$ is stably finite.
%, whence by the linearization trick $\rho$ extends to a $\ast$-embedding of $A\otimes_{\min} B$ into the stably finite C$^*$-algebra $\prod_\om N_k$.  It follows that $A\otimes_{\min}B$ is stably finite.
%Again by \cite[Proposition 4.2]{traces} $(N_k)_\om$ admits a finite AW$^*$-algebra quotient with faithful quasitrace, thus $A\otimes B$ is stably finite as it is simple.
\end{proof}

We need the following fact, which is \cite[Theorem 2.4]{traces}:

\begin{fact}\label{hasatrace}
For an arbitrary unital \cstar-algebra $A$, $A$ has a trace if and only if $A\otimes C_r^*(\mathbb{F}_\infty)$ is stably finite.
\end{fact}

The following follows immediately from the Proposition \ref{stablyfinitetensor} and Fact \ref{hasatrace}.
%, and the main result of \cite{HT}.

\begin{cor}
If $A$ is unital, simple, and has good stably finite nuclear witnesses, then $A$ has a trace.
\end{cor}

\begin{cor}
If every stably finite \cstar-algebra has good stably finite nuclear witnesses, then every stably finite \cstar-algebra has a trace.  
\end{cor}

\begin{rmk}
If Question \ref{jepsf} has a positive answer, then there is a shorter model-theoretic proof of the previous corollary.  Indeed, by the assumption that every stably finite \cstar-algebra has good stably finite nuclear witnesses together with the fact that $\k$ has the joint embedding property, we have that there is a nuclear, stably finite \cstar-algebra $A$ that is locally universal for $\k$.  By the main result of \cite{traces}, $A$ has a trace.  It follows that $A^\omega$ has a trace, whence so does every stably finite \cstar-algebra, being a subalgebra of $A^\omega$.
\end{rmk}

First, recall the definition of a quasitrace from Definition \ref{defqt} above. It follows from \cite{handelman, blackhand} that all stably finite \cstar-algebras admit a quasitrace.  A famous question of Kaplansky asks whether or not every quasitrace on a stably finite \cstar-algebra is  linear, that is, is actually a trace.  By \cite[section II]{blackhand}, this question reduces to the question of whether every AW$^*$ II$_1$-factor is a II$_1$-factor.  It follows that a positive answer to the Kaplansky question implies that all stably finite \cstar-algebras admit a trace.  The main result of \cite{traces} is that every quasitrace on an exact \cstar-algebra is a trace.

Let $\k_\tau$ denote the universally axiomatizable class of structures $(A,\tau)$, where $A\in \k$ and $\tau$ is a quasitrace on $A$.  Given $(A,\tau)\in \k$, we say that $(A,\tau)$ has \emph{good quasitracial stably finite nuclear witnesses} if, in the definition of good stably finite nuclear witnesses, the stably finite algebras $Q_n$ can be equipped with quasitraces $\tau_n$ such that $\rho|_{C^*(F)}: C^*(F)\to \prod_{\omega} Q_n$ is quasitrace-preserving.
By further applying the reasoning of \cite[Lemma 5.10 and Theorem 5.11]{traces} we obtain:

\begin{cor} If $A$ is unital and simple, $\tau$ is an extremal quasitrace on $A$, and $(A,\tau)$ admits good quasitracial stably finite nuclear witnesses, then $\tau$ is a trace.
\end{cor}

\begin{cor}\label{quasitracesaretraces}
If every element of $\k_\tau$ has good quasitracial stably finite nuclear witnesses, then every quasitrace on a stably finite \cstar-algebra is a trace.
\end{cor}

We now show that quasitraces can be put into a model-theoretic framework.  In order to do this, we first recall that every quasitrace on a unital \cstar-algebra is Lipshitz continuous with Lipshitz constant $\sqrt{2}$ by \cite[Corollary II.2.5(iii)]{blackhand}.  Thus, we may consider the language of unital \cstar-algebras augmented by a new predicate symbol $\tau$ whose modulus of uniform continuity is $\Delta_\tau(\epsilon)=\epsilon/\sqrt{2}$.

\begin{prop}
The class of structures $(A,\tau)$, where $A$ is a unital \cstar-algebra and $\tau$ is a quasitrace on $A$, is universally axiomatizable.
\end{prop} 

\begin{proof}
Items (1) and (2) in the definition of quasitrace are easily seen to be universally axiomatizable.  In order to axiomatize item (3), we use an ``approximate-near'' version of (3).  Indeed, it is shown in \cite[Corollary II.2.6]{blackhand} that for any quasitrace $\tau$ on a \cstar-algebra $A$ and for any $\epsilon>0$, there is a $\delta>0$ such that for all self-adjoint $a,b\in A$ with $\|ab-ba\|<\delta$, one has $|\tau(a+b)-\tau(a)-\tau(b)|<\epsilon$.  This result is proven by using an ultrapower construction for dimension functions.  However, the same arguments apply to yield an ultra\emph{product} construction for dimension functions, whence the proof of \cite[Corollary II.2.6]{blackhand} can be adapted to show that the choice of $\delta$ depends only on $\epsilon$ and not on the structure $(A,\tau)$.  Consequently, axiom (3) in the definition of quasitrace can be expressed by
$$\sup_{x,y\in A_{\operatorname{sa}}}\min\left(\delta\dotminus \|xy-yx\|,|\tau(x+y)-\tau(x)-\tau(y)|\dotminus \epsilon\right)=0,$$ where $r\dotminus s:=\max(r-s,0)$.  

Finally, item (4) of the definition of quasitrace can be axiomatized using the quantifier-free definability of $M_2(A)$ in $A^{\operatorname{eq}}$; see \cite[Lemma 4.2.3]{munster}.
\end{proof}

\begin{question}
Does the class $\k_\tau$ have the joint embedding property?
\end{question}

Once again, if the previous question has a positive answer, then there is a shorter, model-theoretic proof of Corollary \ref{quasitracesaretraces}.  Indeed, the assumption of good quasitracial stably finite nuclear witnesses allows us to find $(A,\tau)\in \k_\tau$ that is existentially closed for $\k_\tau$ and with $A$ nuclear.  By the main result of \cite{traces}, $\tau$ is a trace.  Since every element of $\k_\tau$ embeds into the ultrapower of $A$ in a quasitrace-preserving way, the desired result follows.

\end{document}